\documentclass[12pt,reqno]{amsart}

\usepackage{amsthm}
\usepackage{amsmath}
\usepackage{amssymb}
\usepackage[left=2.7cm,top=2.7cm,right=2.7cm,bottom=2.7cm]{geometry}
\usepackage[usenames]{color}
\usepackage{enumerate}
\usepackage[normalem]{ulem}
\usepackage{graphicx}

\begin{document}
\newtheorem{theorem}{Theorem}[section]
\newtheorem{lemma}[theorem]{Lemma}
\newtheorem{definition}[theorem]{Definition}
\newtheorem{conjecture}[theorem]{Conjecture}
\newtheorem{proposition}[theorem]{Proposition}
\newtheorem{claim}[theorem]{Claim}
\newtheorem{algorithm}[theorem]{Algorithm}
\newtheorem{corollary}[theorem]{Corollary}
\newtheorem{observation}[theorem]{Observation}
\newtheorem{problem}[theorem]{Open Problem}
\newtheorem{remark}[theorem]{Remark}
\newcommand{\noin}{\noindent}
\newcommand{\ind}{\indent}
\newcommand{\om}{\omega}
\newcommand{\ff}{\mathcal F}
\newcommand{\pp}{\mathcal P}
\newcommand{\cc}{{\mathcal C}_{\le cn}}
\newcommand{\cco}{{\mathcal C}_{\le n}}
\newcommand{\bAC}{\overline{\AC}}
\newcommand{\ppp}{\mathfrak P}
\newcommand{\N}{{\mathbb N}}
\newcommand{\Z}{{\mathbb Z}}
\newcommand{\LL}{\mathbb{L}}
\newcommand{\R}{{\mathbb R}}
\newcommand{\E}{\mathbb E}
\newcommand{\Prob}{\mathbb{P}}
\newcommand{\eps}{\varepsilon}
\newcommand{\ram}[1]{\hat{R}({#1})}
\newcommand{\G}{{\mathcal{G}}}

\newcommand{\Ss}{{\mathcal S}}
\newcommand{\Nn}{{\mathcal N}}

\newcommand{\ceil}[1]{\left \lceil #1 \right \rceil}
\newcommand{\floor}[1]{\left \lfloor #1 \right \rfloor}
\newcommand{\size}[1]{\left \vert #1 \right \vert}
\newcommand{\dist}{\mathrm{dist}}

\title{Size-Ramsey numbers of cycles versus a path}

\author{Andrzej Dudek}
\address{Department of Mathematics, Western Michigan University, Kalamazoo, MI, USA}
\thanks{The first author was supported in part by Simons Foundation Grant \#244712 and by the National Security Agency under Grant Number H98230-15-1-0172. The United States Government is authorized to reproduce and distribute reprints notwithstanding any copyright notation hereon.}
\email{\tt andrzej.dudek@wmich.edu}

\author{Farideh Khoeini}
\address{Department of Mathematical Sciences, Isfahan University of Technology, Isfahan, Iran}
\email{\tt  f.khoeini@math.iut.ac.ir}

\author{Pawe\l{} Pra\l{}at}
\address{Department of Mathematics, Ryerson University, Toronto, ON, Canada 
and 
The Fields Institute for Research in Mathematical Sciences, Toronto, ON, Canada}
\thanks{The third author was supported by NSERC and Ryerson University}
\email{\tt pralat@ryerson.ca}

\begin{abstract}
The size-Ramsey number $\ram{\ff,H}$ of a family of graphs $\ff$ and a graph $H$ is the smallest integer $m$ such that there exists a graph $G$ on $m$ edges with the property that any colouring of the edges of $G$ with two colours, say, red and blue, yields a red copy of a graph  from $\ff$ or a blue copy of $H$. In this paper we first focus on $\ff = \cc$, where $\cc$ is the family of cycles of length at most $cn$, and $H = P_n$. In particular, we show that $2.00365 n \le \ram{\cco,P_n} \le 31n$. Using similar techniques, we also managed to analyze $\ram{C_n,P_n}$, which was investigated before but only using the regularity method.
\end{abstract}

\maketitle

\section{Introduction}

Following standard notations, for any family of graphs $\ff$ and any graph $H$, we write $G\to (\ff, H)$ if any colouring of the edges of $G$ with 2 colours, red and blue, yields a red copy of some graph from $\ff$ or a blue copy of $H$. For simplicity, we write $G\to (F,H)$ if $\ff = \{F\}$ and $G\to F$ instead of $G\to (F,F)$. We define the \emph{size-Ramsey number} of the pair $(\ff,H)$ as 
$$
\ram{\ff,H} = \min \{ |E(G)| : G \to (\ff,H) \}
$$ 
and again, for simplicity, $\ram{F,H}=\ram{\{F\},H}$ and $\ram{F} = \ram{F,F}$. 

One of the most studied directions in this area is the size-Ramsey number of $P_n$, a path on $n$ vertices. It is obvious that $\ram{P_n} = \Omega(n)$ and that $\ram{P_n} = O(n^2)$ (for example, $K_{2n}\to P_n$), but the exact behaviour of $\ram{P_n}$ was not known for a long time. In fact, Erd\H{o}s~\cite{E81} offered \$100 for a proof or disproof that 
\[
\ram{P_n} / n \to \infty \quad \text{ and } \quad \ram{P_n} / n^2 \to 0.
\]
This problem was solved by Beck~\cite{B83} in 1983 who, quite surprisingly, showed that $\ram{P_n} < 900 n$. (Each time we refer to inequality such as this one, we mean that the inequality holds for sufficiently large $n$.) A variant of his proof, provided by Bollob\'{a}s~\cite{B01}, gives $\ram{P_n} < 720 n$. Very recently, different and more elementary arguments were used by the first and the third author of this paper~\cite{DP15,DP16}, and by Letzter~\cite{L15} that show that $\ram{P_n} < 137 n$~\cite{DP15}, $\ram{P_n} < 91 n$~\cite{L15}, and $\ram{P_n} < 74 n$~\cite{DP16}.
On the other hand, the first nontrivial lower bound was provided by Beck~\cite{B90} and his result was subsequently improved by Bollob\'as~\cite{bollobas2} who showed that $\ram{P_n} \ge (1+\sqrt{2})n - O(1)$; today we know that $\ram{P_n} \ge 5n/2 - O(1)$~\cite{DP16}.

\smallskip

For any $c \in \R_+$, let $\cc$ be the family of cycles of length at most $cn$. In this paper, we continue to use similar ideas as in~\cite{DP15,L15,DP16} to deal with $\ram{\cc,P_n}$. Such techniques (very simple but quite powerful) were used for the first time in~\cite{Krivelevich1, Krivelevich2}; see also recent book~\cite{Krivelevich_book} that covers several tools including this one. Corresponding theorems use different approaches and different probability spaces that might be interesting on their own rights. Some non-trivial lower bounds are provided as well. In particular, it is shown that 
\[
2.00365 n \le \ram{\cco,P_n} \le 31n
\]
for sufficiently large~$n$.

\smallskip

We also study $\ram{C_n,P_n}$ and show that for even and sufficiently large $n$ we have
\[
5n/2 - O(1) \le \ram{C_n,P_n} \le 2257n.
\]
(In fact, the lower bound holds for odd values of $n$, too.) The linearity of $\ram{C_n,P_n}$ also follows from the earlier result of Haxell, Kohayakawa and \L{}uczak~\cite{HKL} who proved that the size-Ramsey number $\ram{C_n,C_n}$ is linear in $n$. However, their proof is based on the regularity method and therefore the leading constant is enormous.

\section{Preliminaries}

Let us recall a few classic models of random graphs that we study in this paper. The \emph{binomial random graph} $\G(n,p)$ is the random graph $G$ with vertex set $[n] := \{ 1, 2, \ldots, n\}$ in which every pair $\{i,j\} \in \binom{[n]}{2}$ appears independently as an edge in $G$ with probability~$p$. 
The \emph{binomial random bipartite graph} $\G(n,n, p)$ is the random bipartite graph $G=(V_1 \cup V_2, E)$ with partite sets $V_1, V_2$, each of order $n$, in which every pair $\{i,j\} \in V_1 \times V_2$ appears independently as an edge in $G$ with probability~$p$. 
Note that $p=p(n)$ may (and usually does) tend to zero as $n$ tends to infinity.  

Recall that an event in a probability space holds \emph{asymptotically almost surely} (or \emph{a.a.s.}) if the probability that it holds tends to $1$ as $n$ goes to infinity. Since we aim for results that hold a.a.s., we will always assume that $n$ is large enough. 

\smallskip

Another probability space that we are interested in is the probability space of random $d$-regular graphs with uniform probability distribution. This space is denoted $\mathcal{G}_{n,d}$, and asymptotics are for $n\to\infty$ with $d\ge 2$ fixed, and $n$ even if $d$ is odd.

Instead of working directly in the uniform probability space of random regular graphs on $n$ vertices $\mathcal{G}_{n,d}$, we use the \textit{pairing model} (also known as the \textit{configuration model}) of random regular graphs, first introduced by Bollob\'{a}s~\cite{bollobas1}, which is described next. Suppose that $dn$ is even, as in the case of random regular graphs, and consider $dn$ points partitioned into $n$ labelled buckets $v_1,v_2,\ldots,v_n$ of $d$ points each. A \textit{pairing} of these points is a perfect matching into $dn/2$ pairs. Given a pairing $P$, we may construct a multigraph $G(P)$, with loops allowed, as follows: the vertices are the buckets $v_1,v_2,\ldots, v_n$, and a pair $\{x,y\}$ in $P$ corresponds to an edge $v_iv_j$ in $G(P)$ if $x$ and $y$ are contained in the buckets $v_i$ and $v_j$, respectively. It is an easy fact that the probability of a random pairing corresponding to a given simple graph $G$ is independent of the graph, hence the restriction of the probability space of random pairings to simple graphs is precisely $\mathcal{G}_{n,d}$. Moreover, it is well known that a random pairing generates a simple graph with probability asymptotic to $e^{-(d^2-1)/4}$ depending on $d$, so that any event holding a.a.s.\ over the probability space of random pairings also holds a.a.s.\ over the corresponding space $\mathcal{G}_{n,d}$. For this reason, asymptotic results over random pairings suffice for our purposes. For more information on this model, see, for example, the survey of Wormald~\cite{NW-survey}.

\smallskip

Also, we will be using the following well-known concentration inequality. Let $X \in \textrm{Bin}(n,p)$ be a random variable with the binomial distribution with parameters $n$ and $p$. Then, a consequence of Chernoff's bound (see, for example, Theorem ~2.1 in~\cite{JLR}) is that for any $t \ge 0$
$$
\Prob( X \le \E X - t) \le \exp \left( - \E X \varphi \left( \frac {-t}{\E X} \right) \right) \le \exp \left( - \frac {t^2}{2\E X} \right),  
$$
where $\varphi(x) = (1+x) \log(1+x) - x$. 

\smallskip

For simplicity, we do not round numbers that are supposed to be integers either up or down; this is justified since these rounding errors are negligible to the asymptomatic calculations we will make. Finally, we use $\log n$ to denote natural logarithms.

\section{Upper bounds---first approach}

In this section, we will use the following observation.

\begin{lemma}\label{lem:determ_upperbound1}
Let $c \in \R_+$ and let $G$ be a graph of order $(c+1)n$. Suppose that for every two disjoint sets of vertices $S$ and $T$ such that $|S| = |T| = cn/2$ we have $e(S,T) \ge cn$. Then, $G\to (\cc,P_n)$.
\end{lemma}

\begin{proof}
Let $G=(V,E)$ be any graph of order $(c+1)n$ with the desired expansion properties. Consider any red-blue colouring of the edges of a graph, and suppose that there is no blue copy of~$P_n$. Our goal is to show that there must be a red cycle of length at most $cn$. 

We perform the following algorithm on $G$ to construct a blue path $P$. Let $v_1$ be an arbitrary vertex of $G$, let $P=(v_1)$, $S = V \setminus \{v_1\}$, and $T = \emptyset$. We investigate all edges from $v_1$ to $S$ searching for a blue edge. If such an edge is found (say from $v_1$ to $v_2$), we extend the blue path as $P=(v_1,v_2)$ and remove $v_2$ from $S$. We continue extending the blue path $P$ this way for as long as possible. Since there is no monochromatic $P_n$, we must reach a point of the process in which $P$ cannot be extended, that is, there is a blue path from $v_1$ to $v_k$ ($k < n$) and there is no blue edge from $v_k$ to $S$. This time, $v_k$ is moved to $T$ and we try to continue extending the path from $v_{k-1}$, reaching another critical point in which another vertex will be moved to $T$, etc. If $P$ is reduced to a single vertex $v_1$ and no blue edge to $S$ is found, we move $v_1$ to $T$ and simply restart the process from another vertex from $S$, again arbitrarily chosen. 

An obvious but important observation is that during this algorithm there is never a blue edge between $S$ and $T$. Moreover, in each step of the process, the size of $S$ decreases by 1 or the size of $T$ increases by 1. Finally, since there is no monochromatic $P_n$, the number of vertices of the blue path $P$ is always smaller than $n$. Hence, at some point of the process both $S$ and $T$ must have size at least $cn/2$. Now, if needed, we remove some vertices from $S$ or $T$ so that both sets have size precisely $cn/2$. Finally, it follows from the expansion property of $G$ that $e(S,T) \ge cn$. Thus, $G[S\cup T]$ (the graph induced by $G \cup T$), which consists only of red edges, is not a forest and so must contain a red cycle. The proof is finished.
\end{proof}

Since random graphs are good expanders, after carefully selecting parameters, they should arrow the desired graphs and so they should provide some upper bounds for the corresponding size Ramsey numbers. Let us start with investigating binomial random graphs. 

\begin{lemma}\label{lem:random1}
Let $c \in \R_+$ and let $d = d(c) > 4/c$ is such that 
$$
(c+1)\log(c+1) + c \log(d/2) - c^2d/4 + c = 0.
$$ 
Then, the following two properties hold a.a.s.\ for $G \in \G((c+1)n,d/n)$:
\begin{enumerate}[(i)]
\item for every two disjoint sets of vertices $S$ and $T$ such that $|S| = |T| = cn/2$ we have $e(S,T) > cn$;
\item $|E(G)| \sim d(c+1)^2 n / 2$.
\end{enumerate}
\end{lemma}
\begin{proof}
Let $S$ and $T$ with $|S| = |T| = cn/2$ be fixed and let $X=X_{S,T}=e(S,T)$.
Clearly, $\E X = c^{2}d n/4 > cn$ and by Chernoff's bound
\begin{eqnarray*}
\Pr(X \le cn) &\le& \exp\left( - \E X \left( \frac {cn}{\E X} \log \left( \frac {cn}{\E X} \right) + \frac {\E X - cn}{\E X} \right) \right) \\
&=& \exp \Big( ( c \log (cd/4) - c^2 d/4 + c ) n \Big).
\end{eqnarray*}
Thus, by the union bound over all choices of $S$ and $T$ we have
\begin{align*}
\Pr\left( \bigcup_{S,T} ( X_{S,T} \le cn ) \right) &\le
\binom{(c+1)n}{cn/2} \binom{(c+1)n-cn/2}{cn/2} \exp \Big( ( c \log (cd/4) - c^2 d/4 + c ) n \Big) \\
&= \frac{((c+1)n)!}{(cn/2)! (cn/2)! n!} \exp \Big( ( c \log (cd/4) - c^2 d/4 + c ) n \Big).
\end{align*}
Using Stirling's formula ($x! \sim \sqrt{2\pi x} (x/e)^x$) we get
\begin{align*}
\Pr & \left( \bigcup_{S,T} (X_{S,T} \le cn) \right) \\
& = o \left(\exp \Big( \left( (c+1)\log (c+1) - c \log(c/2) + c \log(cd/4) - c^{2}d/4 + c \right) n \Big) \right) \\
& = o \left(\exp \Big( \left( (c+1)\log (c+1) + c \log(d/2) - c^{2}d/4 + c \right) n \Big) \right) = o(1),
\end{align*}
by the definition of $d$. Part (i) follows from the first moment method.

Part (ii) follows immediately from Chernoff's bound as the expected number of edges in $\G((c+1)n,d/n)$ is $\binom{(c+1)n}{2} d/n \sim d(c+1)^2n/2$.
\end{proof}

Getting numerical upper bounds for size Ramsey numbers is a straightforward implication of the previous two lemmas. We get the following result.

\begin{theorem}\label{thm:random1}
Let $c \in \R_+$ and let $d = d(c) > 4/c$ is such that 
$$
(c+1)\log(c+1) + c \log(d/2) - c^2d/4 + c = 0.
$$ 
Then, a.a.s.\ $G \in \G((c+1)n,d/n) \to (\cc,P_n)$. As a result, for any $\eps > 0$ and sufficiently large $n$, 
$$
\ram{\cc,P_n} < U_1 = U_1(c) := \left( \frac {d(c+1)^2}{2} + \eps \right) n.
$$
In particular, for sufficiently large $n$, it follows that $\ram{\cco,P_n} < 37 n$ and so by monotonicity $\ram{\cc,P_n} < 37 n$ for any $c \ge 1$. Moreover, for $c \in (0,1]$ we have 
$$
\ram{\cc,P_n} < U_2 = U_2(c) := \frac {80 \log(e/c)}{c} \ n.
$$
\end{theorem}
\begin{proof}
The first part is an immediate consequence of Lemma~\ref{lem:determ_upperbound1} and Lemma~\ref{lem:random1}. In order to get an explicit upper bound $U_2=U_2(c)$, we will show that $d = d(c) < \hat{d} = \hat{d} := 40 \log(e/c) / c$. From this the result will follow since $d(c+1)^2/2 \le 2d < 2\hat{d} = 80 \log(e/c) / c$. (Observe that this bound is of the right order as it is easy to see that $d = \Omega(\log(e/c)/c)$ as $c \to 0$.)

Note that 
\begin{eqnarray*}
f(c,\hat{d}) &=& (c+1)\log(c+1) + c \log(\hat{d}/2) - c^2\hat{d}/4 + c \\
&\le & (c+1) c + c \log(20 / c^2) - 10 c \log(e/c) + c,
\end{eqnarray*}
as for $c \in (0,1]$ we have $\log (c+1) \le c$ and $\log(e/c) \le 1/c$ (and so $\hat{d} \le 40/c^2$). Now, since clearly $\log(e/c) \ge 1$,
\begin{eqnarray*}
f(c,\hat{d}) &\le& 3 c \log(e/c) + c \log(20 / c^2) - 10 c \log(e/c) \\
&<& 3 c \log(e/c) + c \log(e^3 / c^3) - 10 c \log(e/c) = -4 c \log(e/c) < 0.
\end{eqnarray*}
It follows that $d < \hat{d}$, and the proof is complete.
\end{proof}

Some numerical values are calculated in Table~\ref{tab:U1} and presented on Figure~\ref{fig:U1}(a).

\begin{table}[h]
\begin{center}
  \begin{tabular}{|c||c|c|c|c|c|c|c|c|c|c|c|c}
    \hline
    $c$ & $1.0$ & $0.9$ & $0.8$ & $0.7$ & $0.6$ & $0.5$ & $0.4$ & $0.3$ & $0.2$ & $0.1$ \\ \hline \hline
    $d(c)$ & 18.43 & 20.9 & 24.05 & 28.20 & 33.89 & 42.11 & 54.91 & 77.21 & 124.51 & 279.54 \\
    $U_1(c)$ & 37 & 38 & 39 & 41 & 44 & 48 & 54 & 66 & 90 & 170\\
    $U_2(c)$ & 80 & 99 & 123 & 156 & 202 & 271 & 384 & 588 & 1044 & 2643\\
    \hline
  \end{tabular}
\caption{Numerical upper bounds $U_1(c)$ and explicit ones $U_2(c)$ for $\ram{\cc,P_n}$.\label{tab:U1}}
\end{center}
\end{table}

\begin{figure}[htbp]
\begin{tabular}{ccc}
\includegraphics[width=0.3\textwidth]{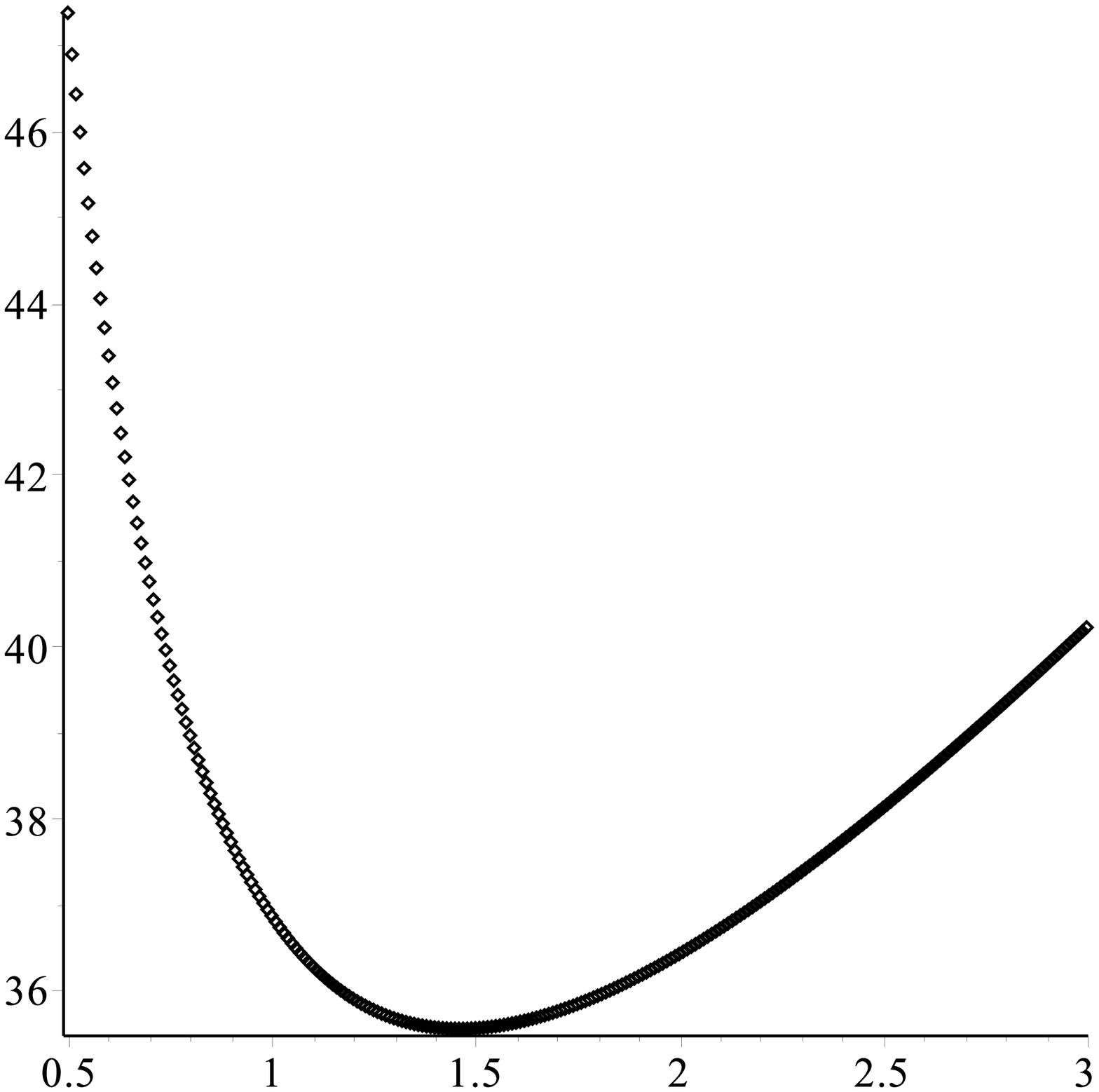} & \quad \quad \quad &
\includegraphics[width=0.3\textwidth]{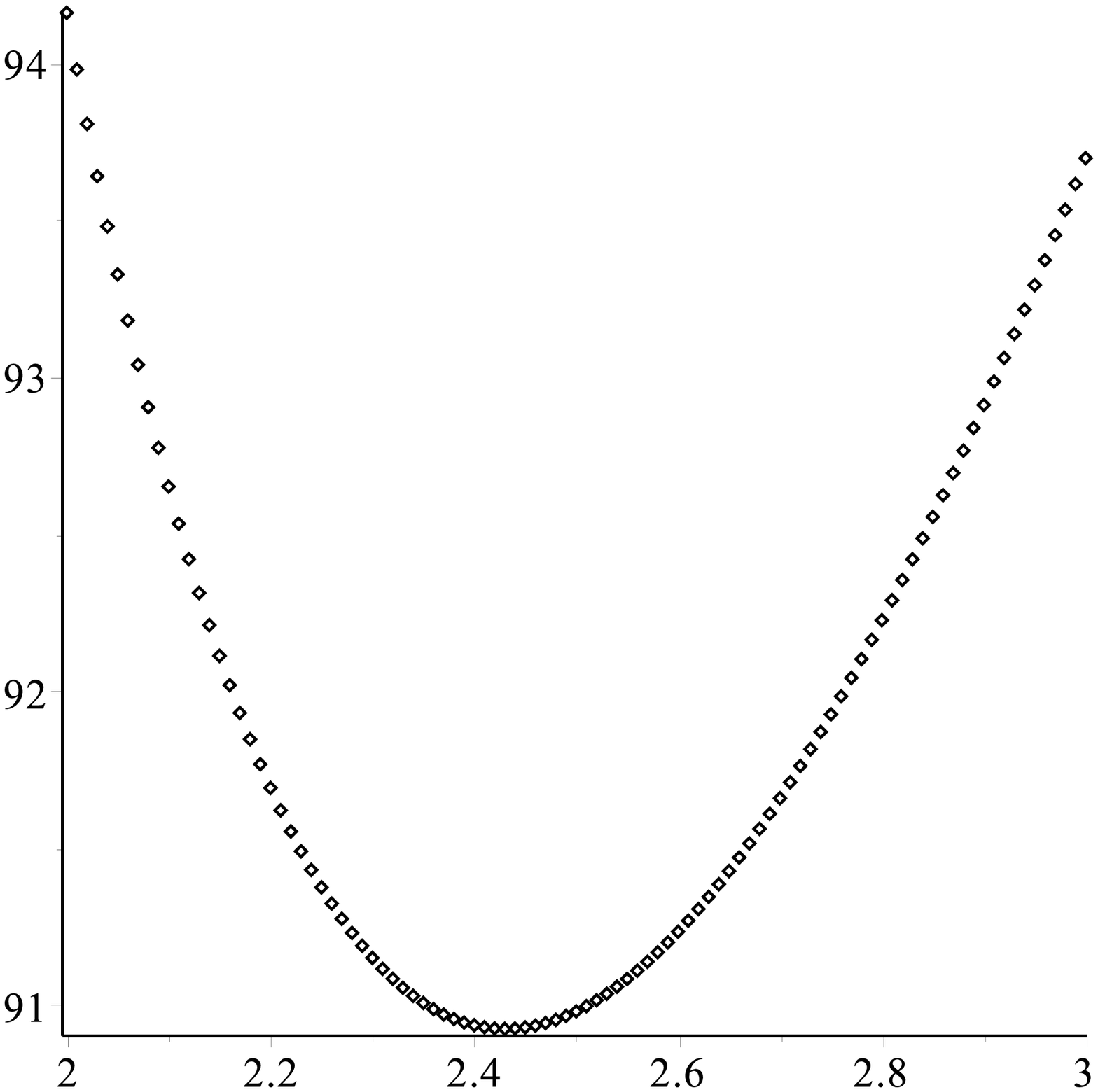} \\
(a) & & (b) \\
\end{tabular}
\caption{Numerical upper bounds for $\ram{\cc,P_n}$: $U_1(c)$ (a) and $U_3(c)$ (from Theorem~\ref{thm:random2}) (b).\label{fig:U1}}
\end{figure}

Now, let us investigate random $d$-regular graphs. For simplicity we focus on the $c=1$ case. Not surprisingly, this model yields slightly better upper bound for the size Ramsey numbers. 

\begin{theorem}\label{thm:d-reg1}
A.a.s.\ $\G_{2n,31} \to (\cco,P_n)$, which implies that $\ram{\cco,P_n} \le 31 n$ for sufficiently large $n$.
\end{theorem}
\begin{proof}
Consider $\mathcal{G}_{2n,d}$ for some $d \in \N$. Our goal is to show that (for a suitable choice of $d$) the expected number of pairs of two disjoint sets, $S$ and $T$, such that $|S|=|T|=n/2$ and $e(S,T) < n$ tends to zero as $n \to \infty$. This, together with the first moment principle, implies that a.a.s.\ no such pair exists and so, by Lemma~\ref{lem:determ_upperbound1}, we get that a.a.s.\ $\mathcal{G}_{2n,d} \to (\cc,P_n)$. As a result, $\ram{\cc, P_n} \le dn$. 

Let $a=a(n)$ be any function of $n$ such that $an \in \Z$ and $0\le a\le 1$ and $b=b(n)$ be any function of $n$ such that $bn \in \Z$ and $0 \le b \le d/2-a$. Let $X(a,b)$ be the expected number of pairs of two disjoint sets $S, T$ such that $|S|=|T|=n/2$, $e(S,T) = an$, and $e(S,V \setminus (S \cup T)) = bn$. Using the pairing model, it is clear that 
\begin{eqnarray*}
X(a) &=& {2n \choose n} {n \choose n/2} {dn/2 \choose an}^2 (an)! {dn/2-an \choose bn} {dn \choose bn} (bn)! \\
&& \quad \cdot \ M(dn/2-an-bn) \cdot M \Big( (dn/2-an) + (dn-bn) \Big) \Big/ M(2dn),
\end{eqnarray*}
where $M(i)$ is the number of perfect matchings on $i$ vertices, that is, 
$$
M(i) = \frac {i!} {(i/2)! 2^{i/2}}.
$$
(Each time we deal with perfect matchings, $i$ is assumed to be an even number.) After simplification we get
\begin{eqnarray*}
X(a,b) &=& (2n)! (dn/2)!^2 (dn)! (3dn/2-an-bn)! (dn)! 2^{dn} \\
&& \quad \cdot \ \Bigg[ n! (n/2)!^2 (an)! (dn/2-an)! (bn)! (dn-bn)! (dn/4-an/2-bn/2)! \\
&& \quad \quad \quad 2^{dn/4-an/2-bn/2}  (3dn/4-an/2-bn/2)! 2^{3dn/4-an/2-bn/2} (2dn)! \Bigg]^{-1}.
\end{eqnarray*}
Using Stirling's formula ($i! \sim \sqrt{2\pi i} (i/e)^i$) and focusing on the exponential part we obtain
$$
X(a,b) = \Theta( n^{-2} ) e^{f(a,b,d)n},
$$
where
\begin{eqnarray*}
f(a,b,d) &=& (3-3d+a+b) \log 2 + d \log d - a \log a - b \log b - (d/2-a) \log (d/2-a) \\
&& \quad - (d-b)\log(d-b) -(d/4-a/2-b/2) \log (d/4-a/2-b/2) \\
&& \quad - (3d/4-a/2-b/2) \log (3d/4-a/2-b/2) \\
&& \quad + (3d/2-a-b) \log (3d/2-a-b).
\end{eqnarray*}
Thus, if $f(a,b,d) \le -\eps$ (for some $\eps > 0$) for any pair of integers $an$ and $bn$ under consideration, then we would get $\sum_{an} \sum_{bn} X(a,b) = O(1) e^{-\eps n} = o(1)$ (as $an = O(n)$ and $bn = O(n)$). The desired property would be satisfied, and the proof would be finished.

It is straightforward to see that 
\begin{eqnarray*}
\frac{\partial f}{\partial b} &=& \log 2 - \log b + \log (d-b) + \log(d/4-a/2-b/2)/2 \\
&& \quad + \log(3d/4-a/2-b/2)/2 - \log(3d/2-a-b).
\end{eqnarray*}
Now, since $\frac{\partial f}{\partial b} = 0$ if and only if 
$$
b^2 - b(2d-2a) + d(d-2a)/2 = 0,
$$ 
function $f(a,b,d)$ has a local maximum for $b = b_0 := d-a-\sqrt{2d^2-4ad+4a^2}/2$, which is also a global one on $b\in(-\infty, d/2-a)$. (Observe that since $b \le d/2 - a$, $b_0 = d/2-a+d/2 - \sqrt{d^2 + (d-2a)^2}/2 \le d/2-a$.) Consequently,
$$
f(a,b,d) \le g(a,d) := f(a, b_0, d).
$$ 
Finally, let us fix $d_0 = 31$. It is easy to show that $g(a,d_0)$ is an increasing function of $a$ on the interval $0 \le a \le a_0=1$. Thus, we get $g(a,d_0) \le g(a_0,d_0) < -0.02 =: -\eps$ and the proof is finished. 
\end{proof}

\section{Upper bounds---second approach}

In this section, we will use another observation.

\begin{lemma}\label{lem:determ_upperbound2}
Let $c > 1$ and let $G$ be a graph of order $2(cn-1)$. Suppose that for every two disjoint sets of vertices $S$ and $T$ such that $|S| = |T| = ((c-1)n-1)/2$ we have $e(S,T) \neq 0$. Then, $G\to (\cc,P_n)$.
\end{lemma}

In order to prove the lemma, we will need the following result of Erd\H{o}s, Faudree, Rousseau, and Schelp~\cite{EFRS}.
Before we state the result, we need to recall a classic counterpart of the size-Ramsey numbers. The \emph{Ramsey number} of the pair $(\ff,H)$ is defined as 
$$
R(\ff,H) = \min \{ n \in \N : K_n \to (\ff,H) \}.
$$ 
Now, we are ready to state the theorem and then prove Lemma~\ref{lem:determ_upperbound2}.

\begin{theorem}[\cite{EFRS}]\label{thm:EFRS}
For all $n \ge 2$ and $c \in (1,\infty)$,
$$
R(\cc,K_n) = 
\begin{cases}
2n-1 & \text{ if } cn \ge 2n-1,\\
2n & \text{ if } n < cn < 2n-1.
\end{cases}
$$
\end{theorem}

\begin{proof}[Proof of Lemma~\ref{lem:determ_upperbound2}]
Suppose that there exists a colouring of the edges of a graph $G$ of order $2(cn-1)$ with neither red cycle on at most $cn$ vertices nor blue path on $n$ vertices. This colouring yields a colouring of the edges of $K_{2(cn-1)}$: edges of $K_{2(cn-1)}$ corresponding to the red edges of $G$ stay red; remaining edges of $K_{2(cn-1)}$ are, say, green. As such colouring does not create any red cycle on at most $cn$ vertices, it follows from Theorem~\ref{thm:EFRS} that it must create a green $K_{cn-1}$. Coming back to graph $G$ and its original red-blue colouring, we get that there exists a set $U$ consisting of $cn-1$ vertices in $G$ that induces a graph with no red edge. Performing the algorithm used in the proof of Lemma~\ref{lem:determ_upperbound1} on $G[U]$ we get that there are two disjoint sets of vertices $S$ and $T$ such that $|S| = |T| = ((c-1)n-1)/2$ with no blue edge between $S$ and $T$. As a result, $e(S,T) = 0$ and the proof is finished.
\end{proof}

As usual, let us first see how the observation works for binomial random graphs.

\begin{lemma}\label{lem:random2}
Let $c > 1$ and let 
$$
d = d(c) := \frac {4}{(c-1)^2} \left( (2c) \log (2c) - (c-1) \log \left( \frac {c-1}{2} \right) - (c+1) \log (c+1) \right) .
$$
Then, the following two properties hold a.a.s.\ for $G \in \G(2(cn-1),d/n)$:
\begin{enumerate}[(i)]
\item for every two disjoint sets of vertices $S$ and $T$ such that $|S| = |T| = ((c-1)n-1)/2$ we have $e(S,T) \neq 0$;
\item $|E(G)| \sim 2dc^2 n$.
\end{enumerate}
\end{lemma}

\begin{proof}
Let $X$ be the random variable counting pairs of disjoint sets $S$ and $T$ such that $|S| = |T| = ((c-1)n-1)/2$ and $e(S,T)=0$. Part (i) follows from the first moment method, since
\begin{eqnarray*}
\E X &\le& \binom{2cn}{(c-1)n/2} \binom{2cn - (c-1)n/2}{(c-1)n/2} \left( 1 - \frac {d}{n} \right)^{ ((c-1)n/2)^2 } \\
&\le& \frac {(2cn)!}{((c-1)n/2)!^2 ((c+1)n)!} \exp \Big( -d (c-1)^2 n /4 \Big)\\
&=& o \left( \exp \left( \left( (2c) \log (2c) - (c-1) \log \left( \frac{c-1}{2} \right) - (c+1) \log (c+1) -\frac {d (c-1)^2}{4} \right) \right) \right) \\
&=&o(1),
\end{eqnarray*}
by the definition of $d$. 

The expected number of edges is asymptotic to
$
\binom{2cn}{2} \frac {d}{n} \sim 2c^2dn,
$
and part (ii) follows from Chernoff's bound.
\end{proof}

Using the above lemmas, we get the following upper bond for the size Ramsey numbers. 

\begin{theorem}\label{thm:random2}
Let $c > 1$ and let 
$$
d = d(c) := \frac {4}{(c-1)^2} \left( (2c) \log (2c) - (c-1) \log \left( \frac {c-1}{2} \right) - (c+1) \log (c+1) \right) .
$$
Then, a.a.s.\ $G \in \G(2(cn-1),d/n) \to (\cc,P_n)$. As a result, for any $\eps > 0$ and sufficiently large $n$, 
$$
\ram{\cc,P_n} < U_3 = U_3(c) := \left( 2c^2d + \eps \right) n.
$$
\end{theorem}

Some numerical values are presented on Figure~\ref{fig:U1}(b). Note that $U_3(c)$ is not a decreasing function of $c$; in particular, $U_3(c) > U_3(2.5)$ for any $c > 2.5$ and so stronger bound for large values of $c$ can be obtained by monotonicity: for $c \ge 2.5$ we get $\ram{\cc,P_n} \le U_3(2.5) < 91n$. In any case, unfortunately, the bound obtained in the previous section (using binomial random graphs) is stronger than the one obtained here for \emph{any} value of $c > 1$.

\bigskip

As before, slightly better bounds are obtained for random $d$-regular graphs. We omit tedious calculations here as the very same optimization problem was considered in~\cite{DP16} (see Theorem~3.2). The goal there was to minimize the number of edges in a random $d$-regular graph of order $c'n$ (for some $c'>2$) with the property (holding a.a.s.) that no two disjoint sets of size $n(c'-2)/4$ have no edge between them. Now, we have the exact same goal with $c=c'/2 > 1$. From the result in~\cite{DP15} it follows that the best bound we get is for (approximately) $c = 5.219/2 = 2.6095$ and $d=30$: a.a.s.\ $\G_{2(cn-1),d} \to (\cc,P_n)$ giving us an upper bound of $78.3 n$ for $c \ge 2.6095$, worse than $31n$ for $c \ge 1$ following from Theorem~\ref{thm:d-reg1}.

\section{Lower bounds}

We start with an easy lower bound.
\begin{theorem}\label{thm:lb}
Let $c \in \R_+$. Then,
$$
\ram{\cc,P_n} \ge 2(n-1).
$$
\end{theorem}
\begin{proof}
Let $G$ be any graph such that $G \to (\cc,P_n)$. Since one can independently colour each connected component, we may assume that $G$ is connected. Let $T$ be any spanning tree of $G$. Colour the edges of $T$ red and blue otherwise. Clearly, there is no red cycle (of any length). Thus, there must be a blue path on $n$ vertices. This implies that $|V(G)| = |V(T)| \ge n$ and the number of blue edges is at least $n-1$. Hence, $|E(G)| \ge |V(T)|-1 + (n-1) \ge 2(n-1)$, and so the result holds.
\end{proof}

We will soon improve the leading constant 2. But first, in order to prepare the reader for more complicated argument, we show a weaker result which improves this constant for graphs with bounded maximum degree.
\begin{theorem}\label{thm:lb2}
Let $c\in \R_+$ and let $G$ be a graph with maximum degree $\Delta$ such that $G \to (\cc,P_n)$. Then, 
\[
|E(G)| \ge n(2+1/\Delta^2)-2.
\]
\end{theorem}

\begin{proof}
Let $G=(V,E)$ be a connected graph with maximum degree $\Delta$ such that $G \to (\cc,P_n)$. We will start by showing that $|V| \ge n(1 + 1/\Delta^2)$. Consider $G^{2}$ (recall that two vertices are adjacent in $G^2$ if they are at distance at most $2$ in $G$). Clearly, the maximum degree of $G^{2}$ is at most $\Delta^2$. Moreover, observe that any independent set $A$ in $G^{2}$ induces a forest between $A$ and $V\setminus A$ in $G$ (in fact, a collection of disjoint stars as no vertex from $V\setminus A$ is adjacent to more than one vertex from $A$ in $G$). Finally, clearly there is an independent set~$A$ in $G^2$ of size at least $|V|/(\Delta^2+1)$.

Now, let us colour the edges of $G$ as follows: colour red all edges between $A$ and $V\setminus A$, and blue otherwise. Obviously there is no red cycle and so there must be a blue path on $n$ vertices. Such path must be entirely contained in $G[V\setminus A]$ as $G[A]$ is an empty graph. Thus, $|V\setminus A| \ge n$ and we get
\[
|V| = |A| + |V\setminus A| \ge |V|/(\Delta^2+1) + n
\]
implying that $|V| \ge n(1 + 1/\Delta^2)$, as required.

The rest of the proof is straightforward. We recolour $G$ as in the proof of Theorem~\ref{thm:lb} obtaining 
\[
|E| \ge |V|-1 + (n-1) = n(2+1/\Delta^2)-2.
\]
The bound holds.
\end{proof}

Using the ideas from the above proof we improve Theorem~\ref{thm:lb}. The improvement of the leading constant might seem negligible. However, it was not clear if one can move away from the constant 2. The observation below answers this question. 

\begin{theorem}
Let $c \in \R_+$. Then for sufficiently large $n$ we have
$$
\ram{\cc,P_n} \ge 2.00365 n.
$$
\end{theorem}

\begin{proof}
Set $a=2.0037$, $b=0.5$, and $d=9$. Suppose that $G=(V,E)$ is a graph such that  $G \to (\cc,P_n)$. Clearly, $G$ has at least $n$ vertices. Let us put all vertices of degree at least $d+1$ to set $B$. We may assume that $B$ contains at most $2a/(d+1)$ fraction of vertices; otherwise, $G$ would have more than $an$ edges and we would be done. Let $A \subseteq V\setminus B$ be an independent set in $G^2$ (and so also in $G$) as in the proof of Theorem~\ref{thm:lb2}. That means the graph induced between $A$ and $V\setminus (A\cup B)$ is a collection of disjoint stars. Furthermore, we may assume that $A$ is maximal (that is, no vertex from $V\setminus (A\cup B)$ can be added to $A$ without violating this property). As in the previous proof we notice that $A$ contains at least $(1- \frac {2a}{d+1})/(d^2+1)$ fraction of vertices of $G$. Now, we need to consider two cases.

\smallskip

\textbf{Case 1:} $|B| \le b |A|$. Colour the edges between $A$ and $V\setminus (A\cup B)$ red and blue otherwise. Clearly, there is no red cycle and so there must be a blue path $P_n$. Moreover, since $|B| \le b |A|$, $(1-b)|A|$ vertices are not part of a blue $P_n$. Thus, 
\[
|V| \ge n + (1-b)|A| \ge n + (1-b)\left( \left(1- \frac {2a}{d+1}\right)/(d^2+1)\right) |V|
\]
yielding
\[
|V| \ge \left( 1 - \frac {(1-b)(1-2a/(d+1)}{d^2+1} \right)^{-1} n.
\]
Finally, we recolour $G$ as in the proof of Theorem~\ref{thm:lb} obtaining 
\[
|E| \ge |V|-1 + (n-1) 
\ge \left( 1 + \left( 1 - \frac {(1-b)(1-2a/(d+1)}{d^2+1} \right)^{-1} \right) n-2 > 2.00366 n
\]
for sufficiently large $n$.

\smallskip

\textbf{Case 2:} $|B| > b |A|$. First colour the edges between $A$ and $V\setminus (A\cup B)$ red. Then, extend the graph induced by the red edges to maximal forest in $G[V\setminus B]$; remaining edges  colour blue. Since $A$ is maximal, $G[V\setminus B]$ consists of at most $|A|$ components and so the number of red edges is at least $n - |B| - |A|$. As in the previous case there is no red cycle and so there exists a blue $P_n$. The number of blue edges that are not on such blue $P_n$ is at least $|B| (d+1-2)/2$. Consequently, the total number of edges is at least
\begin{align*}
(n - |B| - |A|) + (n-1) + |B| (d-1)/2 &\ge 2n - |A| + |B| \frac {d-3}{2} -1\\
&\ge 2n + |A| \left( \frac {b(d-3)}{2} - 1 \right) -1\\
&\ge \left( 2 + \frac {1- \frac {2a}{d+1}}{d^2+1} \left( \frac {b(d-3)}{2} - 1 \right) \right) n -1\\
& > 2.00365 n
\end{align*}
for $n$ large enough.
This completes the proof.
\end{proof}

\section{Size-Ramsey of $C_n$ versus $P_n$}

The lower bound follows immediately from the result obtained by the first and the third author of this paper~\cite{DP16}:
\[
\ram{C_n,P_n} \ge \ram{P_n,P_n} \ge 5n/2 - O(1).
\]
Let us then focus on the upper bound.

\begin{theorem}
For all even and sufficiently large $n$,
\[
\ram{C_n,P_n} \le 2257n.
\]
\end{theorem}
\begin{proof}
First we will estimate the probability of having a cycle $C_n$ in $\G(cn,cn,d_1/n) \cup \G(cn,cn,d_2/n)$.
In order to avoid technical problems with events not being independent, we use a classic technique known as \emph{two-round exposure} (known also as \emph{sprinkling} in the percolation literature). The observation is that a random graph $G \in G(cn,cn,d/n)$ can be viewed as a union of two independently generated random graphs $G_1 \in G(cn,cn,d_1/n)$ and $G_2 \in G(cn,cn,d_2/n)$, provided that $d/n=d_1/n + d_2/n - d_1 d_2/n^2$ (see, for example,~\cite{B01,JLR} for more information). 

Using the algorithm we exploit extensively in this paper (used for the first time in the proof of Lemma~\ref{lem:determ_upperbound1}), it follows that the probability that $\G(cn,cn,d_1/n)$ has no path of length $3n$ is at most
\begin{align}\label{ub:c1}
\binom{cn}{(c-3/2)n/2}^2& \left( 1 - \frac {d_1}{n} \right)^{ \big( (c-3/2) n/2 \big)^2 } \le \notag\\
& \quad \exp \bigg{(} \bigg{(} 2 c \log c - (c-3/2) \log \left( \frac {c-3/2}{2} \right) \notag\\
 &\qquad\qquad - (c+3/2) \log \left( \frac {c+3/2}{2} \right) - \frac {(c-3/2)^2 d_1}{4} \bigg{)} n \bigg{)}.
\end{align}
(Note that a path oscillates between the two partite stets but we do not know how the rest is partitioned. However, regardless how they are partitioned we are always guaranteed to have two sets of size $(c-3/2)n/2$ with no edge between.) 
Now let us assume that a path $(v_1, \ldots, v_{3n/2})\cup (u_1, \ldots, u_{3n/2})$ of length $3n$ was already found in $\G(cn,cn,d_1/n)$. Let us concentrate on two middle vertices $v_{3n/4}$ and $u_{3n/4}$ that we assume belong to the same partite set, and let us fix an even $i \in [n/4, 3n/4]$. We want to construct a cycle of the desired length as follows: $v_{3n/4}$ to some $v_\ell$ along the first path ($\ell < 3n/4$), to a specific $u_L$ ($L < 3n/4$), to $u_{3n/4}$ along the second path, continue to $u_{R}$ for some $R > 3n/4$, to a specific $v_r$, and go back to $v_{3n/4}$. We want the `left' half cycle to be of length $i$ (that is, $3n/4-\ell + 1 + 3n/4-L = i$), and the `right' half to be of length $n-i$ (that is, $r - 3n/4 + 1 + R - 3n/4 = n-i$). This guarantees that for different values of $i$ we always investigate disjoint set of edges. The remaining edges of the cycle, that is $\{v_\ell,u_L\}$ and $\{u_R, v_r\}$, will come from $\G(cn,cn,d_2/n)$, independently generated. Observe that we fail to find an edge on both sides with probability at most
$$
(1-d_2/n)^i + (1-d_2/n)^{n-i} \le \exp(-d_2i/n) + \exp(-d_2(n-i)/n) \le2 \exp(-d_2 / 4),
$$
since $i \in [n/4, 3n/4]$. Now, as we have $n/4$ independent events for various values of $i$ (recall that $i$ must be even), we fail to close a cycle with probability at most 
\begin{equation}\label{ub:c2}
(2 \exp(-d_2/4))^{n/4} = \exp \Big( ((\log 2)/4 - d_2/16) n \Big).
\end{equation}

Now, consider $\G((2c+1)n,d_1/n) \cup \G((2c+1)n,d_2/n) = \G( (2c+1)n, (d_1+d_2-o(1))/n )$, and assume that there is no blue $P_n$. Then, using the algorithm one more time, we get that there are two sets $S$ and $T$ with $|S|=|T|=cn$ such that all edges between $S$ and $T$ are red. Furthermore, due to~\eqref{ub:c1} and~\eqref{ub:c2} the probability that $G[S\cup T]$ contains no copy of $C_n$ is at most
\begin{align*}
& \exp \bigg{(} \bigg{(} 2 c \log c - (c-3/2) \log \left( \frac {c-3/2}{2} \right) - \notag\\
&\qquad\qquad\qquad(c+3/2) \log \left( \frac {c+3/2}{2} \right) - \frac {(c-3/2)^2 d_1}{4} \bigg{)} n \bigg{)} +\exp\Big( ((\log 2)/4 - d_2/16) n/4 \Big).
\end{align*}
On the other hand, the union bound over all choices of $S$ and $T$ contributes only
$$
\binom{(2c+1)n}{cn} \binom{(c+1)n}{cn} = \frac {((2c+1)n)!}{ (cn)!^2 n! } = o(1) \cdot \exp \left( \Big( (2c+1) \log (2c+1) - 2 c \log c \Big) n \right)
$$
number of terms. Since the number of edges present is a.a.s.\ 
$$
\binom{(2c+1)n}{2} (d_1+d_2-o(1))/n \sim \frac {(2c+1)^2}{2} (d_1+d_2) n,
$$
our goal is to minimize $(2c+1)^2 (d_1+d_2) / 2$, provided that
\[
(2c+1) \log (2c+1) - (c-3/2) \log \left( \frac {c-3/2}{2} \right) - (c+3/2) \log \left( \frac {c+3/2}{2} \right) - \frac {(c-3/2)^2 d_1}{4} \le 0
\]
and
\[
(2c+1) \log (2c+1) - 2 c \log c + (\log 2)/4 - d_2/16\le 0.
\]
One can easily check that for $c=2.21$, $d_1 = 60.34$, and $d_2=93.26$ the above inequalities hold and $(2c+1)^2 (d_1+d_2) / 2 < 2257$. 
\end{proof}

\end{document}